\documentclass[a4paper,12pt]{amsart}
\usepackage{amsmath,amscd,amssymb, amsbsy,eufrak,euscript}
\usepackage{mathrsfs}
\def\r{{\mathbb R}^d}
\def\rr{{\mathbb R}^2}
\def\mm#1{\left[#1 \right]}
\def\LT#1{\mathcal L[#1]}
\def\LTL#1{\mathcal L\biggl[#1\biggr]}

\def\ILT#1{{\mathcal L}^{-1}[#1]}

\def\dis{\displaystyle}
\def\L#1#2{L_{#1}^{(#2)}}
\def\vg{\varGamma}
\def\ip#1#2{\langle #1,#2 \rangle}
\def\vl{\operatorname{Vol}}
\def\n#1#2{\rho_#1^{(#2)}}
\def\b#1#2{\beta_{#1,#2}}
\def\vd#1#2{\varSigma_v^{#1}(#2)}
\def\vdl#1{\varXi_v^{#1}(t)}
\def\a{\alpha_{\lambda,v}}

  \newtheorem{thm}{Theorem}
  
  \newtheorem{lemma}[thm]{Lemma}
  \newtheorem{prop}[thm]{Proposition}

  \newtheorem{remark}[thm]{Remark}

\numberwithin{equation}{section}
\numberwithin{thm}{section}

\begin{document}

\title{A formula for the expected volume of\\ the Wiener sausage
with constant drift}

\author{Yuji Hamana}

\address{Department of Mathematics, Kumamoto University,
Kurokami 2-39-1, Kumamoto 860-8555, Japan}
\email{hamana@kumamoto-u.ac.jp}

\author{Hiroyuki Matsumoto}

\address{Department of Physics and Mathematics, Aoyama Gakuin University,
Fuchinobe 5-10-1, Sagamihara 252-5258, Japan}
\email{matsu@gem.aoyama.ac.jp}



\thanks{This work is partially supported by the Grant-in-Aid
for Scientific Research (C) No.24540181 and No.26400144
of Japan Society for the Promotion of Science (JSPS)}

\begin{abstract}
We consider the Wiener sausage for a Brownian motion with a constant drift
up to time $t$ associated with a closed ball.
In the two or more dimensional cases, we obtain the explicit form
of the expected volume of the Wiener sausage. The result says that
it can be represented by the sum of the mean volumes of
the multi-dimensional Wiener sausages
without a drift. In addition, we show that the leading term
of the expected volume of the Wiener sausage is written as
$\kappa t(1+o\mm{1})$ for large $t$ and an explicit form gives
the constant $\kappa$. The expression is of a complicated form,
but it converges to the known constant as the drift tends to $0$.
\end{abstract}

\maketitle

\pagestyle{myheadings}
\markboth{Y. HAMANA and H. MATSUMOTO}{THE WIENER SAUSAGE WITH CONSTANT DRIFT}

%111111111111111111111111111111111111111111111111

\section{Introduction}

\noindent
The figure obtained in moving a non-polar compact set $A$
along the trajectory of a stochastic process $X=\{X(t)\}_{t\geqq0}$
on time interval $[0,t]$, denoted by $W(t;A,X)$, is called the Wiener sausage
for $X$ associated with $A$ up to time $t$. The case that the stochastic process $X$
is a Brownian motion $B=\{B(t)\}_{t\geqq0}$ has been investigated
for a long time and the works in the early stage
are connected with heat conduction problems. In \cite{12}
the trace of the fundamental solution of a heat equation
on a cube with random holes is investigated under some Dirichlet boundary condition.

The expected volume of $W(t;A,B)$ is interpreted
as the total energy flow from a set $A$. Its large time asymptotic
is discussed in \cite{21} and, namely, it is asymptotically equal to
$2\pi t/\log t$ if $d=2$ and $t$ multiple of the Newtonian capacity
of $A$ if $d\geqq3$. In addition, the improved results
are given in \cite{17, 18}.
In the cases that $X$ is a stable process, the same problem is discussed
in \cite{3, 20}. When a compact set $A$ is a closed ball $D$ in $\r$ with
radius $r$, the explicit form and the large time asymptotics
of the expected volume of $W(t;D,B)$ are obtained in \cite{5, 6, 10}.
Moreover, its asymptotic expansion for large $t$ is given in \cite{7}.

We should mention that the limiting behavior of the family of
the volume of $W(t;A,B)$ have been studied.
The laws of large numbers are discussed in \cite{13, 15},
the central limit theorem are in \cite{16}
and the results concerning large deviations are given in
\cite{1, 2, 8}. 

This article deals with the case that $A=D$ and $X$ is a Brownian motion
with a constant drift $v$, denoted by $B_v$, for dimensions two or more.
We show a formula for the expected volume of $W(t;D,B_v)$ in terms of
the mean volumes of the Wiener sausages of different dimensions.
By using the formula, we study the asymptotic behavior of
the expected volume of $W(t;D,B_v)$ and show that it is of the form
$\kappa t(1+o\mm{1})$ for large $t$, $\kappa$ being the constant explicitly written.
As $v$ tends to $0$, the constant
converges to $0$ if $d=2$ and the Newtonian capacity of $D$ if $d\geqq3$.

This article is organized as follows. Section 2 is devoted to
giving the explicit form of the mean volume of $W(t;D,B_v)$ and
Section 3 deals with the calculation of its Laplace transform
which is given in Section 2. In Section 4 we discuss
the principal term of the expected volume of $W(t;D,B_v)$
and its asymptotics as $v$ tends to $0$.

%2222222222222222222222222222222222222222222222

\section{Formula for the expected volume of the Wiener sausage}

\noindent
Let $v$ be a constant vector in $\r$ and $B=\{B(t)\}_{t\geq0}$ be
a standard Brownian motion on $\r$. We consider a Brownian motion
$B_v=\{B_v(t)\}_{t\geqq0}$ with constant drift $v$ given by
$B_v(t)=B(t)+vt$ for $t\geqq0$.
The Wiener sausage $\{W(t;A,B_v)\}_{t\geqq0}$
associated with a compact set $A$
is the set-valued stochastic process defined by
\[
W(t;A,B_v)=\bigcup_{0\leqq s\leqq t}(B_v(s)+A),
\]
where $x+E=\{x+y\,;\,y\in E\}$ for $x\in\r$ and $E\subset \r$.
We consider the case when $A$ is a closed ball
with center $0$ and radius $r>0$, denoted by $D$.
Note that
\[
\begin{split}
\vl (W(t;D,B_v))
&=\vl(\{x\in\r\,;\,|B_v(s)-x|\leqq r\text{ for some $s\in[0,t]$}\})\\
&=\vl(\{x\in\r\,;\,|B_v(s)+x|\leqq r\text{ for some $s\in[0,t]$}\})\\
&=\vl(\{x\in\r\,;\,\inf\{s\geqq0\,;\,|B_v(s)+x|\leqq r\}\leqq t\}),
\end{split}
\]
where $\vl(E)$ is the volume of a set $E$
and $|y|$ is used to denote the Euclidean distance between $0$ and $y$.
For $t\geqq0$ let
\[
\L vd(t)=\int_{\r\setminus D}P_x(\tau_v\leq t)dx,
\]
where $\tau_v=\inf\{t\geqq0\,;\,|B_v(t)|\leqq r\}$
and $P_x$ is the law of the Brownian motion starting from $x\in\r$.
Then we have
\[
E[\vl(W(t;D,B_v))]=\vl(D)+\L vd(t).
\]

The purpose of this section is to give a formula for $\L vd(t)$.
We begin to provide some necessary notation. For $\mu\in\mathbb C$
we use $I_\mu$ and $K_\mu$ for the modified Bessel functions of the first kind
and the second kind, respectively. We often call $K_\mu$ the Macdonald function.
In the case that $\mu$ is real, each modified Bessel function is positive on
the interval $(0,\infty)\subset\mathbb R$.
We refer about the properties of the Bessel functions
to \cite{14, 19, 23}.
For $\mu\geqq0$ and $n\geqq0$ if at least one of them is not $0$, we set
\[
\begin{split}
&\xi_{0,0}=\frac{I_0(r|v|)^2}2,\quad \zeta_{0,0}=\frac{I_0(r|v|)I_1(r|v|)}2,\\
&\xi_{\mu,n}=\frac{(-1)^n (\mu+n)\vg(2\mu+n) I_{\mu+n}(r|v|)^2}{n!},\\
&\zeta_{\mu,n}=\frac{(-1)^n (\mu+n)\vg(2\mu+n) I_{\mu+n}(r|v|)I_{\mu+n+1}(r|v|)}{n!}
\end{split}
\]
Moreover, for $t\geqq0$ and an integer $m\geqq1$ let
\[
\vd mt=e^{-|v|^2t/2} \L 0 m(t)
\]
and
\[
\vdl m=\vd mt+|v|^2\int_0^t \vd ms ds
+\frac14 |v|^4 \int_0^t (t-s) \vd ms ds.
\]
We use $\nu=d/2-1$ and $S_m=2\pi^{(m+1)/2}/\vg((m+1)/2)$
for an integer $m\geqq0$, where $\vg$ is the Gamma function.
It is well-known that $S_m$ coincides with the surface area of
$m$-dimensional unit sphere if $m\geqq1$.

Our main results are the following:

\medskip

\begin{thm}\label{Theorem 2.1}
Let $v\neq0$. For $t>0$ we have that, if $d\geqq2$,
\[
\L vd(t)=\frac{2\pi S_{2\nu}}{|v|^{2\nu}}\sum_{n=0}^\infty
\frac{\xi_{\nu,n}}{S_{2\nu+2n+1}r^{2\nu+2n}} \vdl{d+2n}
+\frac{\pi S_{2\nu} r}{|v|^{2\nu-1}} \biggl(\sum_{n=0}^\infty
\zeta_{\nu,n}\biggr)t.
\]
\end{thm}

\medskip

Throughout this paper the notation $\LT f$ and $\ILT f$
will be used to denoted the Laplace transform and the inverse Laplace transform
of a function $f$, respectively.
In the case when $v=0$, we have the explicit expression for $\L 0d$
by inverting $\LT{\L 0d}$ which have been already
calculated and find that $\L 0d$ is expressed
by modified Bessel functions and their zeros.
See \cite{5, 6, 10} for details.

Similarly to the case when $v=0$, we show Theorem \ref{Theorem 2.1}
by inverting $\LT{\L vd}$
when $v\neq0$. For $\mu\geqq0$ and $\lambda>0$ we put
\[
F_\mu(\lambda)=\frac{\sqrt{2\lambda+|v|^2}}{\lambda^2}
\frac{K_{\mu+1}(r\sqrt{2\lambda+|v|^2})}{K_\mu(r\sqrt{2\lambda+|v|^2})}.
\]
 We now show the explicit form of $\LT{\L vd}$ when $d\geqq2$.

\medskip

\begin{prop}\label{Proposition 2.2}
Let $v\neq0$. For $\lambda>0$ we have that
\begin{equation}
\LT{\L vd}(\lambda)=\frac{\pi S_{2\nu}r}{|v|^{2\nu}}
\sum_{n=0}^\infty \xi_{\nu,n} F_{\nu+n}(\lambda)
+\frac{\pi S_{2\nu}r}{\lambda^2 |v|^{2\nu-1}}
\sum_{n=0}^\infty \zeta_{\nu,n}.
\label{2.1}
\end{equation}
\end{prop}

\medskip

We need to prove that both summations of the right hand side
of \eqref{2.1} converges. The following lemma assures these convergences.

\medskip

\begin{lemma}\label{Lemma 2.3}
For $\mu\geqq0$ and $x>0$,
\[
\sum_{n=1}^\infty |\xi_{\mu,n}|\frac{K_{\mu+n+1}(x)}{K_{\mu+n}(x)}<\infty,\quad
\sum_{n=1}^\infty |\zeta_{\mu,n}|<\infty.
\]
\end{lemma}

\medskip

We prove Lemma \ref{Lemma 2.3} after showing that
Proposition \ref{Proposition 2.2} yields Theorem \ref{Theorem 2.1}.
Moreover the proof of Proposition \ref{Proposition 2.2}
is postponed to the next section.

Since the inverse transform of the second term of \eqref{2.1} is
\[
\frac{\pi S_{2\nu}r}{|v|^{2\nu-1}}
\biggl( \sum_{n=0}^\infty \zeta_{\nu,n}\biggr)t,
\]
it is sufficient to discuss the first term of \eqref{2.1}.
We first invert $F_{m/2-1}$ for each integer $m\geqq2$.
Set $\a=\sqrt{2\lambda+|v|^2}$ throughout this paper.

\medskip

\begin{lemma}\label{Lemma 2.4}
For $t>0$,
\[
\ILT{F_{m/2-1}}(t)=\frac{2 \vdl m}{S_{m-1} r^{m-1}}.
\]
\end{lemma}
\begin{proof}
It is known that, for $\lambda>0$
\[
\LT{\L 0m}(\lambda)=\frac{S_{m-1}r^{m-1}}{\sqrt{2\lambda^3}}
\frac{K_{m/2}(r\sqrt{2\lambda})}{K_{m/2-1}(r\sqrt{2\lambda})}
\]
(cf. \cite{5}). It follows that
\[
\LT{\varSigma_v^m}(\lambda)=\frac{2S_{m-1}r^{m-1}}{\a^3}
\frac{K_{m/2}(r\a)}{K_{m/2-1}(r\a)}
\]
and hence we have
\[
F_{m/2-1}(\lambda)=\frac 2{S_{m-1}r^{m-1}}
\biggl(1+\frac{|v|^2}\lambda+\frac{|v|^4}{4\lambda^2}\biggr)
\LT{\varSigma_v^m}(\lambda).
\]
The standard formulas for the inverse Laplace transform
show the claim of this lemma.
\end{proof}

\medskip

Note that $\nu+n=(d+2n)/2-1$ for each integer $n\geqq0$.
Lemma \ref{Lemma 2.4} yields Theorem \ref{Theorem 2.1}
if we prove
\begin{equation*}
\LTL{\sum_{n=0}^\infty \xi_{\nu,n} \ILT{F_{\nu+n}}}(\lambda)
=\sum_{n=0}^\infty \xi_{\nu,n} F_{\nu+n}(\lambda).
%\label{2.2}
\end{equation*}
For the purpose, we need to show that
\begin{equation}
\sum_{n=1}^\infty |\xi_{\nu,n}|\int_0^\infty
e^{-\lambda t}\ILT{F_{\nu+n}}(t)dt
\label{2.2}
\end{equation}
converges for each $\lambda>0$. Since \eqref{2.2} is equal to
\[
\frac{\a}{\lambda^2}\sum_{n=1}^\infty |\xi_{\nu,n}|
\frac{K_{\nu+n+1}(r\a)}{K_{\nu+n}(r\a)},
\]
Lemma \ref{Lemma 2.3} immediately implies the convergence of \eqref{2.2}.

The remainder of this section is devoted to showing Lemma \ref{Lemma 2.3}.
Before it, we give two lemmas.
One is an upper bound of the ratio of the Macdonald function.

\medskip

\begin{lemma}\label{Lemma 2.5}
For $x>0$ and $\mu\geqq 1/2$
\begin{align}
&K_\mu(x)\leqq 2^{\mu+1}\biggl(2+\frac1x\biggr)^\mu\vg(\mu)e^{-x},\label{2.3}\\
&K_\mu(x)\geqq \frac{2^{\mu-1}}{x^\mu}\vg(\mu) e^{-x}.\label{2.4}
\end{align}
\end{lemma}
\begin{proof}
Recall the integral representation of $K_\mu$:
\begin{equation}
K_\mu(x)=\sqrt{\frac \pi{2x}}\frac{e^{-x}}{\vg(\mu+1/2)}
\int_0^\infty e^{-y}y^{\mu-1/2}\biggl(1+\frac y{2x}\biggr)^{\mu-1/2}dy
\label{2.5}
\end{equation}
(cf. \cite[p.206]{23}). For a proof of \eqref{2.3} we show
\begin{equation}
\int_0^\infty e^{-y}y^{\mu-1/2}\biggl(1+\frac y{2x}\biggr)^{\mu-1/2}dy
\leqq 4\biggl(1+\frac 1{2x}\biggr)^{\mu-1/2} \vg(2\mu).
\label{2.6}
\end{equation}
We divide the integral in the left hand side of \eqref{2.6}
into the following two parts.
One is the integral on $[0,1]$ and the other is that on $(1,\infty)$.
The integral on $[0,1]$ is bounded by
\[
\int_0^1 e^{-y}y^{\mu-1/2}\biggl(1+\frac y{2x}\biggr)^{\mu-1/2}dy
\leqq 
\biggl(1+\frac 1{2x}\biggr)^{\mu-1/2}\int_0^1 e^{-y}dy
\leqq \biggl(1+\frac 1{2x}\biggr)^{\mu-1/2}.
\]
Note that $3\vg(2\mu)\geqq 1$, which can be easily deduced from
$\mu\geqq 1/2$ and
\[
\vg(2\mu)\geqq \int_1^\infty e^{-x}dx\geqq\frac 1e.
\]
Thus we have 
\[
\int_0^1 e^{-y}y^{\mu-1/2}\biggl(1+\frac y{2x}\biggr)^{\mu-1/2}dy\leqq
\biggl(1+\frac 1{2x}\biggr)^{\mu-1/2} 3\vg(2\mu).
\]
Since, for the integral on $(1,\infty)$,
\[
\begin{split}
\int_1^\infty e^{-y}y^{\mu-1/2}\biggl(1+\frac y{2x}\biggr)^{\mu-1/2}dy
&=\int_1^\infty e^{-y}y^{2\mu-1}\biggl(\frac 1y+\frac 1{2x}\biggr)^{\mu-1/2}dy\\
%&\leqq \biggl(1+\frac 1{2x}\biggr)^{\mu-1/2}\int_1^\infty e^{-y}y^{2\mu-1}dy\\
&\leqq \biggl(1+\frac 1{2x}\biggr)^{\mu-1/2}\vg(2\mu).
\end{split}
\]
We obtain \eqref{2.6}.

With the help of the inequality $1/\sqrt{2x}\leqq \sqrt{1+1/2x}$ for $x>0$,
we deduce from \eqref{2.5} and \eqref{2.6} that
\[
K_\mu(x)\leqq
4\biggl(1+\frac 1{2x}\biggr)^\mu e^{-x}\frac{\sqrt\pi \vg(2\mu)}{\vg(\mu+1/2)}.
\]
The well-known formula
\begin{equation}
2^{2z-1}\vg(z)\vg\biggl(z+\frac12\biggr)=\sqrt\pi \vg(2z)
\label{2.7}
\end{equation}
(cf. \cite[p.3]{14}) yields \eqref{2.3}.

For the proof of \eqref{2.4}, note that
\[
\int_0^\infty e^{-y}y^{\mu-1/2}\biggl(1+\frac y{2x}\biggr)^{\mu-1/2}dy=
\int_0^\infty e^{-y}y^{2\mu-1}\biggl(\frac 1y+\frac 1{2x}\biggr)^{\mu-1/2}dy
\geqq\frac{\vg(2\mu)}{(2x)^{\mu-1/2}}.
\]
This, combined with \eqref{2.5} and \eqref{2.7}, implies \eqref{2.4}.
%It is easy to see that \eqref{2.4} follows from \eqref{2.5},
%\eqref{2.7} and this estimate.
\end{proof}

\medskip

The other lemma is concerned with an upper bound of the ratio of Gamma functions.

\medskip

\begin{lemma}\label{Lemma 2.6}
Let $\mu\geqq0$. For an integer $n\geqq1$
\begin{equation}
\frac{\vg(2\mu+n)}{\vg(\mu+n)^2}
\leqq \frac{2^{\mu-1}\vg(2\mu+1)}{\vg(\mu+1)^2}\frac{2^n}{\vg(n)}.
\label{2.8}
\end{equation}
\end{lemma}
\begin{proof}
We prove \eqref{2.8} by induction.
It is easy to see \eqref{2.8} when $n=1$. We assume that \eqref{2.8} holds
for some integer $m\geqq 1$, and then have
\[
\frac{\vg(2\mu+m+1)}{\vg(\mu+m+1)^2}
=\frac{(2\mu+m)\vg(2\mu+m)}{(\mu+m)^2\vg(\mu+m)}
\leqq \frac{2^{\mu-1}\vg(2\mu+1)}{\vg(\mu+1)^2}\frac{(2\mu+m)2^n}{(\mu+m)^2\vg(m)}.
\]
Since $2\mu+m\leqq2(\mu+m)$ and $\mu+m\geqq m$, we obtain \eqref{2.8}
for $n=m+1$.
%This completes the proof of \eqref{2.8}.
\end{proof}

\medskip

We now in a position to show Lemma \ref{Lemma 2.3}.
Let $\mu\geqq0$ and $n\geqq1$. Recall that, for $x>0$
\begin{equation}
I_{\mu+n}(x)\leqq \frac1{\vg(\mu+n+1)}
\biggl(\frac x2\biggr)^{\mu+n} e^x
\label{2.9}
\end{equation}
(cf. \cite{11}). By \eqref{2.3} and \eqref{2.4}, for $x>0$
\begin{equation}
\frac{K_{\mu+n+1}(x)}{K_{\mu+n}(x)}\leqq \frac{8(2x+1)^{\mu+n+1}(\mu+n)}x.
\label{2.10}
\end{equation}
It follows from \eqref{2.9} and \eqref{2.10} that
\begin{equation}
|\xi_{\mu,n}|\frac{K_{\mu+n+1}(x)}{K_{\mu+n}(x)}
\leqq \frac{8(2x+1)\vg(2\mu+n)}{n!\vg(\mu+n)^2 x}
\biggl\{\frac{(2x+1)r^2|v|^2}4\biggr\}^{\mu+n}e^{2r|v|}.
\label{2.11}
\end{equation}
By Lemma \ref{Lemma 2.6} there exists a constant $C$, independent of $n$,
such that
\[
|\xi_{\mu,n}|\frac{K_{\mu+n+1}(x)}{K_{\mu+n}(x)}
\leqq \frac C{n!(n-1)!}\biggl\{\frac{(2x+1)r^2|v|^2}2\biggr\}^n.
\]
This implies the first claim of Lemma \ref{Lemma 2.3}.

In virtue of \eqref{2.9}, we have that
\begin{equation}
|\zeta_{\mu,n}|\leqq \frac{\vg(2\mu+n)}{n!\vg(\mu+n)\vg(\mu+n+2)}
\biggl(\frac{r|v|}2\biggr)^{2\mu+2n+1}e^{2r|v|}.
\label{2.12}
\end{equation}
Since $\mu+n\geqq1$, $\vg(\mu+n+2)\geqq \vg(\mu+n)$.
We apply Lemma \ref{Lemma 2.6} again and
hence the right hand side of \eqref{2.12}
is bounded by a constant, which is independent of $n$, multiple of
\[
\frac 1{n!(n-1)!}\biggl(\frac{r^2|v|^2}2\biggr)^n.
\]
This yields the second claim of Lemma \ref{Lemma 2.3}.

%33333333333333333333333333333333333333333

\section{Laplace transform of $\L vd$}

\noindent
Our goal in this section is to prove Proposition \ref{Proposition 2.2} given in Section 2.
For an integer $n\geqq0$ and a real number $\mu\geqq0$ we write $C_n^\mu$
for the Gegenbauer polynomial. In the case of $\mu=0$ we have that $C_0^0(x)=1$ and
\[
C_n^0(x)=\sum_{m=0}^{[n/2]}\frac{(-1)^m\vg(n-m)}{\vg(m+1)\vg(n-2m+1)}(2x)^{n-2m}
\]
for $n\geqq1$, where $[x]$ is the integer
which is not larger than $x$. In the case of $\mu>0$ we have
\[
C_n^\mu(x)=\frac1{\vg(\mu)}\sum_{m=0}^{[n/2]}
\frac{(-1)^m\vg(\mu+n-m)}{\vg(m+1)\vg(n-2m+1)}(2x)^{n-2m}.
\]
More information on the Gegenbauer polynomials can be found in \cite{19}, for example.

We first give a formula for an improper integral concerning modified Bessel functions.

\medskip

\begin{lemma}\label{Lemma 3.1}
For $0<a<b$, $c>0$ and $\mu\geqq0$ let
\[
R_\mu(a,b;c)=\int_c^\infty x I_\mu(ax) K_\mu(bx) dx.
\]
Then we have
\[
R_\mu(a,b;c)=\frac{ac I_{\mu+1}(ac) K_\mu(bc)+bc I_\mu(ac) K_{\mu+1}(bc)}{b^2-a^2}.
\]
\end{lemma}
\begin{proof}
We refer to \cite{23} about formulas for modified Bessel functions.
Among them, recall
\[
\frac d{dx}x^{-\mu} I_\mu(ax)=ax^{-\mu} I_{\mu+1}(ax),\quad
\frac d{dx}x^{\mu+1}K_{\mu+1}(bx)=-bx^{\mu+1}K_\mu(bx).
\]
Then the integration by parts gives
\[
R_\mu(a,b;c)
=-\frac xb I_\mu(ax) K_{\mu+1}(bx)\biggr|_c^\infty
+\frac ab \int_c^\infty x I_{\mu+1}(ax) K_{\mu+1}(bx)dx.
\]
Since, for each $\eta\geqq0$
\[
K_\eta(x)=\sqrt{\frac\pi{2x}}e^{-x}(1+o[1])\quad\text{and}\quad
I_\eta(x)=\frac1{\sqrt{2\pi x}}e^x(1+o[1])
\]
as $x\to\infty$, we have
\[
R_\mu(a,b;c)=\frac cb I_\mu(ac) K_{\mu+1}(bc)+\frac ab R_{\mu+1}(a,b;c).
\]
Similarly we can deduce that
\[
R_{\mu+1}(a,b;c)=\frac cb I_{\mu+1}(ac) K_\mu(bc)
+\frac ab R_\mu(a,b;c)
\]
from the formulas
\[
\frac d{dx}x^{\mu+1} I_{\mu+1}(ax)=ax^{\mu+1} I_\mu(ax),\quad
\frac d{dx}x^{-\mu}K_\mu(bx)=-bx^{-\mu}K_{\mu+1}(bx).
\]
Thus the assertion of this lemma can be easily obtained.
\end{proof}

\medskip

We are ready to prove Proposition \ref{Proposition 2.2}.
Similarly to Proposition 3.1 in \cite{5},
we can show that, for $\lambda>0$
\[
\LT{\L vd}(\lambda)=\frac 1\lambda
\int_{\r\setminus D} E_x[e^{-\lambda\tau_v}]dx.
\]
Recall that we have put $\nu=d/2-1$ and $\a=\sqrt{2\lambda+|v|^2}$.
Moreover we use $\b xy$ to denote $\ip xy/|x|\cdot |y|$ for $x,y\in\r\setminus\{0\}$.
The explicit form of $E_x[e^{-\lambda \tau_v}]$
given in \cite{11} immediately
yields that $\LT{\L v2}(\lambda)$ is the sum of
\begin{align}
&\frac1\lambda \int_{\rr\setminus D} e^{-\ip vx}I_0(r|v|)
\frac{K_0(|x|\a)}{K_0(r \a)}dx,\label{3.1}\\
&\frac1\lambda \int_{\rr\setminus D} e^{-\ip vx}\sum_{n=1}^\infty
n C_n^0 (\b vx) I_n(r|v|) \frac{K_n(|x|\a)}{K_n(r \a)}dx\label{3.2}
\end{align}
and that, if $d\geqq3$, $\LT{\L vd}(\lambda)$ is equal to
\begin{equation}
\begin{split}
\frac{2^\nu\vg(\nu)}\lambda\int_{\r\setminus D} e^{-\ip vx}\sum_{n=0}^\infty
&(\nu+n)C_n^\nu (\b vx)\\
&\times \frac{I_{\nu+n}(r|v|)}{(r|v|)^\nu}
\frac{|x|^{-\nu}K_{\nu+n}(|x|\a)}{r^{-\nu}K_{\nu+n}(r \a)}dx.
\end{split}
\label{3.3}
\end{equation}

We consider the case when $d\geqq3$ and justify the change of the order of
the integral and the summation in \eqref{3.3}. The dominated convergence theorem
implies that we can change the order once we prove the convergence of
\begin{equation}
\sum_{n=1}^\infty\frac{I_{\nu+n}(r|v|)}{(r|v|)^\nu}
\int_{\r\setminus D} e^{-\ip vx} (\nu+n)
\left| C_n^\nu (\b vx)\right|
\frac{|x|^{-\nu}K_{\nu+n}(|x|\a)}{r^{-\nu}K_{\nu+n}(r \a)}dx.
\label{3.4}
\end{equation}
The following estimate given in \cite{11}
is useful: for $\mu\geqq0$, $n\geqq1$ and $|y|\leqq1$
\begin{equation}
(\mu+n)\left| C_n^\mu (y)\right|
\leqq \frac{4^n \delta_\mu \vg(\mu+n+1)}{n!},
\label{3.5}
\end{equation}
where $\delta_0=1$ and $\delta_\mu=1/\vg(\mu)$ for $\mu>0$. It follows from
\eqref{2.9} and \eqref{3.5} that \eqref{3.4} is dominated by
\begin{equation}
\frac{r^\nu}{2^\nu \vg(\nu)}e^{r|v|}
\sum_{n=1}^\infty \frac{(2r|v|)^n}{n!}
\int_{\r\setminus D} e^{|v|\cdot |x|} |x|^{-\nu}
\frac{K_{\nu+n}(|x|\a)}{K_{\nu+n}(r \a)}dx.
\label{3.6}
\end{equation}
It follows from Lemma \ref{Lemma 2.5} that, for $a\geqq b>0$
and $\mu\geqq1/2$
\[
\frac{K_\mu(a)}{K_\mu(b)}
\leqq 4(2b+1)^\mu e^{-a+b}.
\]
Note that $\a>|v|$ for $\lambda>0$. Then we have that the integral
in \eqref{3.6} is dominated by
\begin{equation}
4(2r\a+1)^{\nu+n} e^{r\a}
\int_{\r\setminus D} |x|^{-\nu}e^{-|x|(\a-|v|)}dx.
\label{3.7}
\end{equation}
The polar coordinate transformation gives that the integral in \eqref{3.7} is
\[
S_{d-1}\int_r^\infty \rho^{\nu+1}e^{-(\a-|v|)\rho}d\rho
\leqq \frac{S_{d-1}\vg(\nu+2)}{(\a-|v|)^{\nu+2}}.
\]
Hence \eqref{3.6} and also \eqref{3.4} converge. This implies that
we can change the order of the integral and the summation
in \eqref{3.3}, and then we obtain that \eqref{3.3} is equal to
\begin{equation}
\begin{split}
\frac{2^\nu\vg(\nu)}{\lambda |v|^\nu}\sum_{n=0}^\infty
(&\nu+n)
\frac{I_{\nu+n}(r|v|)}{K_{\nu+n}(r \a)}\\
&\times \int_{\r\setminus D} e^{-\ip vx}
C_n^\nu (\b vx)|x|^{-\nu}K_{\nu+n}(|x|\a)dx.
\end{split}
\label{3.8}
\end{equation}

We can compute the integral explicitly.
Let $w=(1,0,\dots,0)\in\r$. We take the orthogonal matrix $T$
of order $d$ which satisfies $Tv=|v|w$.
The integral in \eqref{3.8} is equal to
\begin{equation}
\int_{\r\setminus D} e^{-\ip {Tv}{Tx}} C_n^\nu (\b {Tv}{Tx})
|Tx|^{-\nu}K_{\nu+n}(|Tx|\a)dx.
\label{3.9}
\end{equation}
By a change of variables given by $y=Tx$, \eqref{3.9} is equal to
\begin{equation}
\int_{\r\setminus D} e^{-|v|\ip wy} C_n^\nu (\b wy)
|y|^{-\nu}K_{\nu+n}(|y|\a)dy.
\label{3.10}
\end{equation}
We change variables from $(y_1,y_2,\dots,y_d)$ to
the polar coordinate $(\rho,\theta_1,\theta_2,\dots,\theta_{d-1})$
given by
\[
\begin{array}{l}
y_1=\rho \cos\theta_1,\\
y_2=\rho \sin\theta_1 \cos\theta_2,\\
\vdots\\
y_{d-1}=\rho \sin\theta_1 \sin\theta_2\cdots
\sin\theta_{d-2}\cos\theta_{d-1},\\
y_d=\rho \sin\theta_1 \sin\theta_2\cdots
\sin\theta_{d-2}\sin\theta_{d-1},\\
\rho\geqq0,\,\,0\leqq\theta_j\leqq\pi
\,(j=1,2,\dots, d-2),\,\,0\leqq\theta_{d-1}\leqq 2\pi.
\end{array}
\]
Since the Jacobian of this transform is
\[
\frac{\partial(y_1,y_2,y_3,\dots,y_d)}
{\partial(\rho,\theta_1,\theta_2,\dots,\theta_{d-1})}
=\rho^{d-1}\sin^{d-2}\theta_1 \sin^{d-3}\theta_2\cdots \sin\theta_{d-2},
\]
\eqref{3.10} coincides with
\begin{multline*}
\int_r^\infty d\rho\int_0^\pi d\theta_1 \int_0^\pi d\theta_2
\cdots \int_0^\pi d\theta_{d-2}\int_0^{2\pi}d\theta_{d-1}
e^{-|v|\rho\cos\theta_1}C_n^\nu(\cos\theta_1)\\
\times \rho^{\nu+1} K_{\nu+n}(\a \rho)
\sin^{d-2}\theta_1 \sin^{d-3}\theta_2
\cdots \sin\theta_{d-2},
\end{multline*}
which is the product of the following two integrals:
\begin{align}
&\int_r^\infty d\rho \int_0^\pi d\theta_1
e^{-|v|\rho\cos\theta_1}C_n^\nu(\cos\theta_1)\sin^{d-2}\theta_1
\rho^{\nu+1} K_{\nu+n}(\a\rho),
\label{3.11}\\
&\int_0^\pi d\theta_2 \cdots \int_0^\pi d\theta_{d-2}\int_0^{2\pi}d\theta_{d-1}
\sin^{d-3}\theta_2 \cdots \sin\theta_{d-2}.
\label{3.12}
\end{align}

The calculation of \eqref{3.12} is easy. Indeed, the polar coordinate transformation
gives that the volume of the unit ball in ${\mathbb R}^{d-1}$, denoted by $\omega_{d-1}$, is
\[
\begin{split}
&\int_0^1d\rho \int_0^\pi d\theta_1 \cdots
\int_0^\pi d\theta_{d-3}\int_0^{2\pi}d\theta_{d-2}
\rho^{d-2} \sin^{d-3}\theta_1 \cdots \sin\theta_{d-3}\\
&=\frac1{d-1}\int_0^\pi d\theta_1 \cdots
\int_0^\pi d\theta_{d-3}\int_0^{2\pi}d\theta_{d-2}
\sin^{d-3}\theta_1 \cdots \sin\theta_{d-3},
\end{split}
\]
which yields that \eqref{3.12} coincides with $(d-1)\omega_{d-1}=S_{d-2}$.

For \eqref{3.11} we use the following formula:
\[
\int_0^\pi e^{i z\cos\theta}C_n^\mu(\cos\theta) \sin^{2\mu}\theta d\theta
=\frac{2^\mu i^n\sqrt\pi \vg(\mu+1/2)\vg(2\mu+n)}{n!\,\vg(2\mu)}
\frac{J_{\mu+n}(z)}{z^\mu}
\]
for $\mu>0$, $n\geqq1$ and $z\in\mathbb C$ with $|\arg z|<\pi$
(cf. \cite[p.221]{19}),
where $J_\mu$ is the Bessel function of the first kind.
Applying the formula for $z=|v|\rho e^{i\pi/2}$ and $\mu=\nu$, we have that
\[
\int_0^\pi e^{-|v|\rho\cos\theta}C_n^\nu(\cos\theta)\sin^{2\nu}\theta d\theta=
\frac{2^\nu (-1)^n\sqrt\pi \vg(\nu+1/2)\vg(2\nu+n)}{n!\,\vg(2\nu)}
\frac{I_{\nu+n}(|v|\rho)}{(|v|\rho)^\nu}.
\]
Here we have used the relation
\begin{equation}
e^{i\pi\mu/2}I_\mu(x)=J_\mu(xe^{i\pi/2})
\label{3.13}
\end{equation}
for $x>0$ (cf. \cite[p.77]{23}). Hence it follows from \eqref{2.7} that
\eqref{3.11} is
\[
\frac{(-1)^n\pi\vg(2\nu+n)}{2^{\nu-1}n!\,\vg(\nu) |v|^\nu}
\int_r^\infty \rho I_{\nu+n}(|v|\rho)K_{\nu+n}(\a\rho)d\rho,
\]
which yields that \eqref{3.8} and also $\LT{\L vd}(\lambda)$ are equal to
\[
\frac{2\pi S_{d-2}}{\lambda |v|^{2\nu}}\sum_{n=0}^\infty
\frac{(-1)^n (\nu+n)\vg(2\nu+n)}{n!}\frac{I_{\nu+n}(r|v|)}{K_{\nu+n}(r\a)}
R_{\nu+n}(|v|,\a;r).
\]
It follows from Lemma \ref{Lemma 3.1} that
\[
\LT{\L vd}(\lambda)=
\frac{\pi S_{d-2}r}{\lambda^2 |v|^{2\nu}}\sum_{n=0}^\infty
\biggl\{\a\xi_{\nu,n}
\frac{K_{\nu+n+1}(r\a)}{K_{\nu+n}(r\a)}
+|v|\zeta_{\nu,n}\biggr\}.
\]
Lemma \ref{Lemma 2.3} completes the proof of
Proposition \ref{Proposition 2.2} when $d\geqq3$.

We next prove Proposition \ref{Proposition 2.2} in the case of $d=2$.
Since the calculation is similar to \eqref{3.3},
we shall give the outline of the proof.
Details are left to the reader. It is easy to see that we can change
the order of the integral and the summation by applying \eqref{2.9} and \eqref{3.5}.
Then it follows that \eqref{3.2} is
\begin{equation}
\frac 1\lambda \sum_{n=1}^\infty \frac{nI_n(r|v|)}{K_n(r\a)}
\int_{\rr\setminus D}e^{-\ip vx}C_n^0(\b vx)K_n(|x|\a)dx.
\label{3.14}
\end{equation}
We note that $\LT{\L v2}(\lambda)$ is the sum of \eqref{3.1} and \eqref{3.14}.
Similarly to \eqref{3.8}, we apply an orthogonal transform and
the polar coordinate transformation. Then we obtain that \eqref{3.1} is
\begin{equation}
\frac 1\lambda \frac{I_0(r|v|)}{K_0(r\a)}\int_r^\infty d\rho \int_0^{2\pi}d\theta\,
e^{-|v|\rho\cos\theta}K_0(\a\rho)
\label{3.15}
\end{equation}
and that \eqref{3.14} is
\[
\frac1\lambda \sum_{n=1}^\infty \frac{nI_n(|v|r)}{K_n(r\a)}
\int_r^\infty d\rho \int_0^{2\pi}d\theta\,
e^{-|v|\rho\cos\theta}C_n^0(\cos\theta)K_n(\a\rho),
\]
which coincides with
\begin{equation}
\frac2\lambda \sum_{n=1}^\infty \frac{I_n(|v|r)}{K_n(r\a)}
\int_r^\infty d\rho \int_0^{2\pi}d\theta\,
e^{-|v|\rho\cos\theta}\cos(n\theta) K_n(\a\rho).
\label{3.16}
\end{equation}
Here we have used the formula $C_n^0(\cos\theta)=2\cos(n\theta)/n$ for $n\geqq1$
(cf. \cite[p.218]{19}).
Recall that, for $n\geqq0$ and $z\in\mathbb C$ with $|\arg z|<\pi$
$$
\int_0^\pi e^{iz\cos\theta}\cos(n\theta)d\theta=\pi i^n J_n(z)
$$
(cf. \cite[p.79]{19}). Combining this formula with \eqref{3.13},
we obtain
\[
\int_0^{2\pi} e^{-|v|\rho \cos\theta} \cos(n\theta)d\theta
=2(-1)^n \pi I_n(|v|\rho).
\]
Hence \eqref{3.15} and \eqref{3.16} are
\[
\frac{2\pi}\lambda \frac{I_0(r|v|)}{K_0(r\a)}
R_0(|v|,\a;r)
\]
and
\[
\frac{4\pi}\lambda \sum_{n=1}^\infty \frac{(-1)^n I_n(|v|r)}{K_n(r\a)}
R_n(|v|,\a;r),
\]
respectively. Lemmas \ref{Lemma 2.3} and \ref{Lemma 3.1}
give the claim of Proposition \ref{Proposition 2.2} when $d=2$.

%444444444444444444444444444444444444444444

\section{Large time asymptotics}

\noindent
This section deals with the first approximation of $\L vd(t)$ for large $t$.
When $v=0$, we find in \cite{21} that, if $d=2$,
\[
\lim_{t\to\infty}\frac{\log t}t\L 02(t)=2\pi
\]
and that, if $d\geqq3$,
\[
\lim_{t\to\infty}\frac{\L 0d(t)}t=\frac{(d-2)S_{d-1}r^{d-2}}2,
\]
which coincides with the Newtonian capacity of $D$. Moreover
the several smaller terms of $\L 0d(t)$ are given in \cite{6, 7, 10, 17, 18}.

We consider the same problem in the case of $v\neq0$ and put
\[
\n vd=
\begin{cases}
\dis \frac{\pi I_0(r|v|)}{K_0(r|v|)}+2\pi\sum_{n=1}^\infty
\frac{(-1)^n I_n(r|v|)}{K_n(r|v|)}&\quad\text{if $d=2$,}\\
\dis \frac{\pi S_{d-2}}{|v|^{2\nu}}\sum_{n=0}^\infty\frac{(-1)^n(\nu+n)\vg(2\nu+n)}{n!}
\frac{I_{\nu+n}(r|v|)}{K_{\nu+n}(r|v|)}&\quad\text{if $d\geqq3$.}
\end{cases}
\]
It is easy to show that the right hand side converges absolutely.
It follows form \eqref{2.4} and \eqref{2.9} that, for $n\geqq1$
$$
\frac{I_{\nu+n}(r|v|)}{K_{\nu+n}(r|v|)}
\leqq \frac{(r|v|)^{2\nu+2n}}{2^{2\nu+2n-1}\vg(\nu+n)\vg(\nu+n+1)}e^{2r|v|}.
$$
Hence, applying Lemma \ref{Lemma 2.6}, we have
\begin{equation}
\frac{(\nu+n)\vg(2\nu+n)}{n!}\frac{I_{\nu+n}(r|v|)}{K_{\nu+n}(r|v|)}
\leqq \frac{(r|v|)^{2\nu}\vg(2\nu+1)}{2^\nu \vg(\nu+1)^2}
\frac{(r|v|)^{2n}}{2^n n!(n-1)!}e^{2r|v|}.
\label{4.1}
\end{equation}

The result in this section is the following.

\medskip

\begin{thm}\label{Theorem 4.1}
Let $v\neq0$. We have that
\[
\lim_{t\to\infty}\frac{\L vd(t)}t=\n vd.
\]
\end{thm}

\medskip

In principle, we can derive the principal term of $\L vd(t)$
by its explicit form given in Theorem 2.1. However it is difficult
to obtain it since the representation is complicated.
The main tool to see Theorem \ref{Theorem 4.1} here is the Tauberian theorem.
We study asymptotic behavior of $\LT{\L vd}(\lambda)$
for small $\lambda$.

\medskip

\begin{lemma}\label{Lemma 4.2}
Let $v\ne0$. If $d\geqq2$, then
% $\LT{\L vd}(\lambda)=\n vd \lambda^{-2}(1+o\mm{1})$
%as $\lambda\downarrow 0$.
\[
\lim_{\lambda\downarrow 0}\lambda^2 \LT{\L vd}(\lambda)=\n vd.
\]
\end{lemma}
\begin{proof}
Recall that $\LT{\L vd}(\lambda)$ is represented by \eqref{2.1}.
We first observe the asymptotic behavior of
\begin{equation}
\lambda^2 \sum_{n=0}^\infty \xi_{\nu,n} F_{\nu+n}(\lambda).
\label{4.2}
\end{equation}
The definition of $F_\mu$ immediately yields
\[
\lim_{\lambda\downarrow0}\lambda^2 F_{\nu+n}(\lambda)
=\frac{|v|K_{\nu+n+1}(r|v|)}{K_{\nu+n}(r|v|)}.
\]
On the other hand, it follows from \eqref{2.11} and
Lemma \ref{Lemma 2.6} that, for $n\geqq1$
\[
\begin{split}
|\xi_{\nu,n}|\lambda^2  F_{\nu+n}(\lambda)
&=|\xi_{\nu,n}|
\frac{\a K_{\nu+n+1}(r\a)}{K_{\nu+n}(r\a)}\\
&\leqq \frac{4\vg(2\nu+1)e^{2r|v|}}{r\vg(\nu+1)^2}
\frac 1{n!(n-1)!}\biggl\{\frac{(2r\a+1)r^2|v|^2}2\biggr\}^{\nu+n}
\end{split}
\]
Note that $0\leqq \a\leqq a_{1,v}$ for $0<\lambda<1$ and
the right hand side is bounded by 
\begin{equation}
\frac {\kappa_1}{n!(n-1)!}\biggl\{\frac{(2ra_{1,v}+1)r^2|v|^2}2\biggr\}^n
\label{4.3}
\end{equation}
for some positive constant $\kappa_1$, independent of $\lambda$ and $n$.
Since \eqref{4.3} is independent of $\lambda$ and the summation of
\eqref{4.3} on $n$ over $[1,\infty)$ converges, \eqref{4.2} converges to
\[
\sum_{n=0}^\infty \xi_{\nu+n}\frac{|v|K_{\nu+n+1}(r|v|)}{K_{\nu+n}(r|v|)}.
\]
as $\lambda\downarrow0$. This yields that
\[
\lim_{\lambda\downarrow0}\lambda^2\LT{\L vd}(\lambda)
=\frac{\pi S_{2\nu}r}{|v|^{2\nu-1}}
\sum_{n=0}^\infty \xi_{\nu+n}\frac{K_{\nu+n+1}(r|v|)}{K_{\nu+n}(r|v|)}
+\frac{\pi S_{2\nu}r}{|v|^{2\nu-1}} \sum_{n=0}^\infty \zeta_{\nu,n}.
\]
By Lemma \ref{Lemma 2.3} the right hand side is equal to
\begin{equation}
\frac{\pi S_{2\nu}r}{|v|^{2\nu-1}} \sum_{n=0}^\infty
\biggl\{ \xi_{\nu+n}\frac{K_{\nu+n+1}(r|v|)}{K_{\nu+n}(r|v|)}+ \zeta_{\nu,n}\biggr\}.
\label{4.4}
\end{equation}

For the sake of convenience we put
\[
\phi_{\nu,n}=\xi_{\nu+n}
\frac{K_{\nu+n+1}(r|v|)}{K_{\nu+n}(r|v|)}+ \zeta_{\nu,n}.
\]
It follows that, unless $\nu=n=0$,
\[
\begin{split}
\phi_{\nu,n}=&
\frac{(-1)^n (\nu+n)\vg(2\nu+n)}{n!}
\frac{I_{\nu+n}(r|v|)}{K_{\nu+n}(r|v|)}\\
&\hskip8mm \times \{K_{\nu+n+1}(r|v|)I_{\nu+n}(r|v|)+
K_{\nu+n}(r|v|)I_{\nu+n+1}(r|v|)\}.
\end{split}
\]
By the formula
\[
K_{\mu+1}(z)I_\mu(z)+K_\mu(z)I_{\mu+1}(z)=\frac1z
\]
(cf. \cite[p.80]{23}), if $(\nu, n)\neq (0,0)$, then
\[
\phi_{\nu,n}=\frac{(-1)^n (\nu+n)\vg(2\nu+n)}{r|v|n!}
\frac{I_{\nu+n}(r|v|)}{K_{\nu+n}(r|v|)}.
\]
Moreover, in the case of $\nu=n=0$, we similarly have
$$
\phi_{0,0}=\frac1{2r|v|}\frac{I_0(r|v|)}{K_0(r|v|)}.
$$
These immediately imply that \eqref{4.4} is equal to $\n vd$.
\end{proof}

\medskip

The remainder of this section is devoted to give
the limiting value of the leading term of $\L vd(t)/t$
as $|v|$ tends to $0$.

\medskip

\begin{thm}\label{Theorem 4.3}
We have that
\[
\lim_{v\to0}\lim_{t\to\infty}\frac{\L vd(t)}t=\n 0d,
\]
where
\[
\n 0d=
\begin{cases}
0&\quad\text{if $d=2$,}\\
\dis \frac{(d-2)S_{d-1}r^{d-2}}2&\quad\text{if $d\geqq3$.}
\end{cases}
\]
\end{thm}
\begin{proof}
We show that $\n vd$ converges to $\n 0d$ as $|v|$ varnishes
with the help of the dominated convergence theorem.

It is known that, as $x\downarrow0$,
\[
\begin{split}
&I_\mu(x)=\frac{x^\mu}{2^\mu \vg(\mu+1)}(1+o[1]),\\
&K_\mu(x)=
\begin{cases}
\dis \frac1{\log(1/x)}(1+o[1])\quad&\text{if $\mu=0$,}\\
\dis \frac{2^{\mu-1}\vg(\mu)}{x^\mu}(1+o[1])\quad&\text{if $\mu>0$}
\end{cases}
\end{split}
\]
(cf. \cite[p.136]{14}). Hence we have that
\[
\lim_{|v|\to0}\frac{I_{\nu+n}(r|v|)}{|v|^{2\nu}K_{\nu+n}(r|v|)}=
\begin{cases}
\dis\frac{r^{2\nu}}{2^{2\nu-1}\vg(\nu)\vg(\nu+1)}\quad&\text{if $\nu>0$ and $n=0$,}\\
0\quad&\text{otherwise.}
\end{cases}
\]
Moreover we deduce from \eqref{4.1} that, for $0<|v|<1/r$
\[
\frac{(\nu+n)\vg(2\nu+n)}{|v|^{2\nu}n!}\frac{I_{\nu+n}(r|v|)}{K_{\nu+n}(r|v|)}
\leqq\frac{\kappa_2}{2^n n!(n-1)!},
\]
where $\kappa_2$ is a constant which depends on only $r$ and $\nu$.
The dominated convergence theorem yields that
\[
\lim_{v\to0}\n v2=0
\]
and that, if $d\geqq3$,
\[
\lim_{v\to0}\n vd=\frac{\pi S_{d-2} \nu\vg(2\nu) r^{2\nu}}{2^{2\nu-1}\vg(\nu)\vg(\nu+1)}.
\]
It follows from \eqref{2.7} that the right hand side is equal to
\begin{equation}
\frac{\sqrt \pi S_{d-2}\nu\vg(\nu+1/2)}{\vg(\nu+1)}.
\label{4.5}
\end{equation}
Recall that $\nu=d/2-1$ and $S_{m-1}=2\pi^{m/2}{\vg(m/2)}$ for an integer $m\geqq2$.
We can easily obtain that \eqref{4.5} is $(d-2)S_{d-1}r^{d-2}/2$.
\end{proof}

\medskip

We consider $\n vd$ as a function of $v$ on $\r$.
In the similar way as Theorem \ref{Theorem 4.3}, we can see that $\n vd$ is continuous on $\r$.
The calculation is left to the reader.

\medskip

\begin{remark}\label{Remark 4.4}
When $A$ is a non-polar compact set, the law of large numbers of
the volume of $W(t;A,B_v)$ has been already established.
For details see \cite{8, 13, 22}. Theorem \ref{Theorem 4.1} yields that
\[
\lim_{t\to\infty}\frac{\vl(W(t;D,B_v))}t=\n vd
\]
almost surely. 
\end{remark}

%RRRRRRRRRRRRRRRRRRRRRRRRRRRRRRRRRRRRRRRRRR

\end{document}